\documentclass[11pt]{amsart}
\usepackage[T1]{fontenc}
\usepackage{lmodern}
\usepackage{amssymb}
\usepackage{microtype}
\usepackage{aliascnt}
\usepackage[hidelinks]{hyperref}
\usepackage[noabbrev,capitalize]{cleveref}
\usepackage[shortlabels]{enumitem}

\DeclareMathOperator{\Area}{Area}
\DeclareMathOperator{\Ric}{Ric}
\DeclareMathOperator{\tr}{tr}

\numberwithin{equation}{section}

\newtheorem{theorem}{Theorem}[section]
\newaliascnt{lemma}{theorem}
\newtheorem{lemma}[lemma]{Lemma}
\aliascntresetthe{lemma}
\newaliascnt{proposition}{theorem}
\newtheorem{proposition}[proposition]{Proposition}
\aliascntresetthe{proposition}
\newaliascnt{corollary}{theorem}
\newtheorem{corollary}[corollary]{Corollary}
\aliascntresetthe{corollary}

\theoremstyle{remark}
\newaliascnt{remark}{theorem}
\newtheorem{remark}[remark]{Remark}
\aliascntresetthe{remark}

\newcommand{\R}{\mathbb{R}}
\newcommand{\Z}{\mathbb{Z}}
\newcommand{\Sec}{\operatorname{Sec}}
\newcommand{\dd}{\,\mathrm d}

\begin{document}
	
	\title[Stable minimal hypersurfaces]
	{Stable Minimal Hypersurfaces in Positively Curved $4$-Manifolds}
	
	\author{Han Hong}
	\address{Department of Mathematics and Statistics, Beijing Jiaotong University,
		Beijing 100044, China}
	\email{hanhong@bjtu.edu.cn}
	
	\author{Gaoming Wang}
	\address{Beijing Institute of Mathematical Sciences and Applications,
		Beijing 100044, China}
	\email{gaomingwang@bimsa.cn}
	
	\begin{abstract}
		Let $M^3\to X^4$ be a complete, connected, two-sided stable minimal immersion.
		We prove that if the ambient sectional curvature is nonnegative and the
		ambient scalar curvature has a positive uniform lower bound, then $M$ is
		totally geodesic and its normal Ricci curvature vanishes.
		We also construct a complete metric of strictly positive sectional curvature
		on $\R^4$ admitting a complete, embedded, one-ended, nonparabolic, two-sided
		stable minimal hypersurface diffeomorphic to $\R^3$ which is not totally
		geodesic.
		The rigidity proof combines spectral splitting theory, a warped $\mu$-bubble
		construction, and a harmonic function level set argument.
		The example is obtained by a compactly supported deformation of an example in
		\cite{CLS}.
	\end{abstract}
	
	\maketitle
	
	\section{Introduction}\label{sec:introduction}
	
	Let $(X^{n+1},\bar g)$ be a Riemannian $(n+1)$-dimensional manifold and let
	\begin{equation*}
		F:M^n\longrightarrow X^{n+1}
	\end{equation*}
	be a complete, connected, two-sided minimal immersion.
	We denote by $\nu$ a global unit normal and by $A$ the second fundamental
	form, and we set
	\begin{equation}\label{eq:q-definition}
		q:=|A|^2+\overline{\Ric}(\nu,\nu).
	\end{equation}
	The immersion is stable when
	\begin{equation}\label{eq:stability}
		\int_M q\phi^2\leq \int_M|\nabla\phi|^2
		\qquad\text{for every }\phi\in C_c^\infty(M).
	\end{equation}
	Understanding the topology and geometry of such stable hypersurfaces is a
	fundamental problem.
	When $n=2$, the classical theory is well developed.
	In Euclidean three-space, every complete stable minimal surface is a plane by
	the work of Fischer-Colbrie--Schoen \cite{FCS}, do Carmo--Peng
	\cite{doCarmoPeng}, and Pogorelov \cite{Pogorelov}.
	More generally, Fischer-Colbrie--Schoen \cite{FCS} showed that in a
	three-manifold of nonnegative scalar curvature, the possible conformal types
	are the sphere, torus, plane, and cylinder.
	The torus and cylinder cases also yield strong rigidity conclusions: the
	surface is intrinsically flat and totally geodesic (see \cite{Schoen-Yau-PSC,Carlotto-Chodosh-Eichmair-PMT}).
	The higher-dimensional problem is much less understood.
	
	The first major progress in the next dimension was made by Chodosh--Li
	\cite{ChodoshLi} in Euclidean space (see also \cite{chodoshliR4anisotropic,catino,cabre2026gradientestimatesgreenkernel}) and by Chodosh--Li--Stryker \cite{CLS}
	in general Riemannian manifolds.
	In the latter work, authors proved a Bernstein-type theorem for complete two-sided
	stable minimal hypersurfaces $M^3\to X^4$ under the assumptions
	\begin{equation*}
		\overline\Sec\geq0,
		\qquad
		\overline R\geq R_0>0,
	\end{equation*}
	together with \emph{weakly bounded geometry} of the ambient manifold.
	The two curvature assumptions are geometrically meaningful: Examples~1.1
	and 1.2 of Chodosh--Li--Stryker \cite{CLS} show that the full rigidity
	conclusion ($q=0$) can fail once one of them is removed.
	By contrast, weakly bounded geometry plays a more technical role in their
	proof.
	We briefly recall why their approach requires this condition.
	
	Indeed, because
	$q=|A|^2+\overline\Ric(\nu,\nu)\geq0$, the stability inequality would
	immediately imply $q\equiv0$ if one could find compactly supported cutoffs
	$\phi_j\to1$ locally with
	\begin{equation*}
		\int_M|\nabla\phi_j|^2\longrightarrow0.
	\end{equation*}
	Parabolic ends already possess such cutoffs.
	Chodosh--Li--Stryker \cite{CLS} use nonnegative sectional curvature to show
	that there is at most one nonparabolic end.
	On that remaining end, the uniform positive scalar curvature allows the
	construction of a $\mu$-bubble exhaustion whose boundary components have
	uniformly controlled intrinsic diameter.
	The last step is to upgrade this diameter control to a uniform volume bound
	for fixed-width annular regions.
	Chodosh--Li--Stryker obtain the required intrinsic Ricci lower bound from the
	Gauss equation and a uniform estimate for $|A|$; weakly bounded geometry
	enters in the derivation of this curvature estimate.
	Bishop--Gromov volume comparison then gives an almost-linear volume-growth
	estimate, from which one constructs the desired cutoffs and completes the
	stability argument.
	Under nonnegative ambient Ricci curvature, Chodosh--Li--Stryker \cite{CLS}
	also note that weakly bounded geometry may be replaced by a uniform upper
	bound for sectional curvature.
	This approach has been further developed in related problems
	\cite{hong24,hongyan,wuyujie,wuyujie2}, where weakly bounded geometry is also
	assumed.
	
	We remove the weakly bounded geometry assumption.
	
	\begin{theorem}\label{thm:main}
		Let $(X^4,\bar g)$ be a complete Riemannian manifold satisfying
		\begin{equation}
			\overline{\Sec}\geq0,
			\qquad
			\overline R\geq\kappa>0.
		\end{equation}
		Every complete, connected, two-sided stable minimal immersion
		$M^3\to X^4$ satisfies
		\begin{equation}
			A\equiv0,
			\qquad
			\overline{\Ric}(\nu,\nu)\equiv0.
		\end{equation}
		
	\end{theorem}
	
	The proof is intrinsic based on nonnegative curvatures in the spectral sense. First we establish the following two identities
	\begin{equation}
		\Ric_M\geq-\frac23q,
		\qquad
		R_M+2q=\overline R+|A|^2\geq\kappa,
	\end{equation}
	Here $\Ric_M$ denotes the smallest eigenvalue of the Ricci tensor.
	Stability provides a positive Jacobi function $u$ with
	\begin{equation}\label{eq:jacobi-u-intro}
		\Delta u+q u=0.
	\end{equation}
	Consequently the same function satisfies the positive spectral scalar
	inequality
	\begin{equation}
		-\Delta u+\frac12R_Mu\geq\frac\kappa2u.
	\end{equation}
	
	The first ingredient is the sharp spectral splitting theorem of
	Antonelli--Pozzetta--Xu \cite{APX}, combined with the topology theorem of
	Catino--Mari--Mastrolia--Roncoroni \cite{CMMR}.
	In the only nontrivial case, $M$ has one end and $H_2(M;\Z)=0$; in fact,
	$M$ is diffeomorphic to $\R^3$ in this case.
	The second ingredient is a family of separating warped $\mu$-bubbles
	$\Sigma_R$ escaping to infinity.
	Keeping the normally discarded square in the second variation yields
	\begin{equation}
		\int_{\Sigma_R}
		\left(H^2+h^2+|\partial_\nu\log u|^2\right)\leq C
	\end{equation}
	for a constant $C$ independent of $R$.
	The third ingredient is a harmonic function between a fixed inner boundary
	and $\Sigma_R$.
	Integration by parts converts all interior terms into nonnegative squares.
	On $\Sigma_R$, the normal derivative of the weight $u$ cancels exactly with
	the corresponding term in the weighted mean-curvature equation.
	This gives an $R$-independent $L^1$ bound for the gradient of the harmonic
	function.
	Nonparabolicity would then produce a nonconstant function with finite $L^1$
	energy, contradicting conservation of flux.
	This last step is inspired by recent work of Yan--Zhu \cite{YanZhu}.
	Hence $M$ is parabolic, and \eqref{eq:stability} forces $q\equiv0$.
	
	It is then natural to ask whether the uniform scalar curvature gap in
	\Cref{thm:main} can be removed.
	Equivalently, does nonnegative ambient sectional curvature alone force a
	complete two-sided stable minimal $3$-dimensional hypersurface to be totally
	geodesic?
	Cabr\'e--Catino--Mari--Mastrolia--Roncoroni
	\cite[p.~9]{CCMMR26} explicitly raised this question.
	Our second result gives a negative answer; in fact, the counterexample has
	strictly positive sectional curvature.
	
	\begin{theorem}\label{thm:counterexample}
		There exist a complete smooth metric $G$ on $\R^4$ satisfying
		$\Sec_G>0$ and a complete embedded hypersurface
		\begin{equation*}
			\Sigma^3\subset(\R^4,G),
			\qquad
			\Sigma\cong\R^3,
		\end{equation*}
		which is one-ended, nonparabolic, two-sided, stable, and minimal, but is
		not totally geodesic.
	\end{theorem}
	
	We briefly explain the construction.
	Example~1.2 of Chodosh--Li--Stryker \cite{CLS} provides a complete positively
	curved metric $G_0$ on $\R^4$ containing a complete, stable, totally geodesic
	hypersurface $\Sigma\cong\R^3$.
	In Fermi coordinates around a compact subset of $\Sigma$, write
	\begin{equation*}
		G_0=dt^2+g_t^0,
		\qquad
		\Sigma=\{t=0\}.
	\end{equation*}
	Choose a nonzero compactly supported trace-free symmetric tensor $T$ on
	$\Sigma$.
	With a cutoff in the $t$-variable, define
	\begin{equation*}
		g_t^\varepsilon
		=g_t^0+2\varepsilon tT
		+\frac23\varepsilon^2t^2|T|^2g_0.
	\end{equation*}
	The derivatives at $t=0$ give
	\begin{equation*}
		g_\varepsilon|_{T\Sigma}=g_0|_{T\Sigma},
		\qquad
		A_\varepsilon=\varepsilon T,
		\qquad
		H_\varepsilon=0,
	\end{equation*}
	and
	\begin{equation*}
		\Ric_{G_\varepsilon}(\nu,\nu)
		=\Ric_{G_0}(\nu,\nu)-\varepsilon^2|T|^2.
	\end{equation*}
	Consequently the Jacobi potential is preserved pointwise:
	\begin{equation*}
		|A_\varepsilon|^2+\Ric_{G_\varepsilon}(\nu,\nu)
		=\Ric_{G_0}(\nu,\nu).
	\end{equation*}
	Thus the induced metric, completeness, and stability are unchanged, whereas
	$A_\varepsilon\not\equiv0$.
	Since the deformation is compactly supported and $C^2$-small, strict
	positivity of sectional curvature on its support is preserved.
	The scalar curvature of the starting model tends to zero at infinity, so
	\Cref{thm:counterexample} does not contradict \Cref{thm:main}.
	Moreover, the example shows that the uniform scalar gap is a necessary
	rigidity hypothesis and that intrinsic spectral information alone cannot
	detect total geodesicity.
	
	Let us remark that this example has positive Ricci curvature decaying to zero. It is interesting to mention that there exists closed $4$-manifold with positive Ricci curvature which contains a complete stable two-sided minimal hypersurface that is totally geodesic (see \cite{ChodoshLi,maximo-example}).
	
	\vskip.2cm
	The paper is organized as follows.
	In \Cref{sec:hypersurface-identities} we derive the intrinsic Ricci and scalar
	identities associated with the Jacobi function.
	\Cref{sec:topological-reduction} performs the topological reduction.
	\Cref{sec:bubble-separators} constructs separating warped $\mu$-bubbles and
	proves the uniform boundary estimate.
	\Cref{sec:harmonic-identity} establishes the harmonic function identity, and
	\Cref{sec:main-proof} combines these ingredients to prove \Cref{thm:main}.
	Finally, \Cref{sec:counterexample-proof} proves \Cref{thm:counterexample} by
	compactly deforming the first and second normal derivatives of the ambient
	metric in the Chodosh--Li--Stryker example.
	
	\subsection*{Acknowledgments}
	
	The first author is supported by the Fundamental Research Funds for the
	Central Universities, grant no.~YA26JBMC00040, and by the National Natural
	Science Foundation of China, grant no.~12401058.
	
	\subsection*{AI disclosure} The construction of perturbed metric example in \Cref{sec:counterexample-proof} was initially
	provided through interactions with ChatGPT 5.6. The
	authors verified and completed all mathematical arguments and
	take full responsibility for the content. ChatGPT 5.6 is also used for the presentation of this paper.
	
	\section{Hypersurface identities and the spectral reduction}
	\label{sec:hypersurface-identities}
	
	Throughout this section, $M^3\to X^4$ satisfies the assumptions of
	\Cref{thm:main}.
	The condition $\overline\Sec\geq0$ gives $q\geq0$.
	
	\begin{lemma}\label{lem:intrinsic-ricci}
		At every point of $M$,
		\begin{equation}\label{eq:rho-control}
			\Ric_M\geq-\frac23|A|^2\geq-\frac23q.
		\end{equation}
		\begin{equation}\label{eq:scalar-gauss}
			R_M+2q=\overline R+|A|^2.
		\end{equation}
	\end{lemma}
	
	\begin{proof}
		Choose a local frame $e_1,e_2,e_3$ on $M$.
		The Gauss equation gives
		\begin{equation*}
			\Ric_M(e_1,e_1)
			=\sum_{j\geq  1}\overline K(e_1,e_j)-h_{1j}^2.
		\end{equation*}
		Since $M$ is minimal, we have
		\begin{equation*}
			\sum h_{1j}^2\leq\frac23|A|^2.
		\end{equation*}
		This proves \eqref{eq:rho-control}.
		Tracing the Gauss equation and using
		$H=0$ gives
		\begin{equation*}
			R_M=\overline R-2\overline\Ric(\nu,\nu)-|A|^2.
		\end{equation*}
		Combining this with \eqref{eq:q-definition} proves
		\eqref{eq:scalar-gauss}.
	\end{proof}
	
	Fischer-Colbrie--Schoen \cite{FCS} showed that stability yields a function
	$u\in C^\infty(M)$ with $u>0$ satisfying \eqref{eq:jacobi-u-intro}.
	
	\begin{corollary}\label{cor:two-spectral}
		The function $u$ satisfies
		\begin{align}
			-\frac23\Delta u+\Ric_Mu&\geq0,
			\label{eq:spectral-ricci-nonnegative}\\
			-\Delta u+\frac12R_Mu&\geq\frac\kappa2u.
			\label{eq:spectral-biric-positive}
		\end{align}
	\end{corollary}
	
	\begin{proof}
		Since $-\Delta u=q u$, \Cref{lem:intrinsic-ricci} gives
		\begin{equation*}
			-\frac23\Delta u+\Ric_Mu
			=\left(\frac23q+\Ric_M\right)u\geq0.
		\end{equation*}
		The scalar identity gives
		\begin{equation*}
			-\Delta u+\frac12R_Mu
			=\left(q+\frac12R_M\right)u
			=\frac12(\overline R+|A|^2)u
			\geq\frac\kappa2u.
		\end{equation*}
	\end{proof}
	
	In dimension three, $R_M/2$ is the bi-Ricci curvature.
	Thus
	\eqref{eq:spectral-biric-positive} is precisely a positive spectral
	bi-Ricci bound with parameter $\gamma=1$.
	
	\section{Topological reduction}\label{sec:topological-reduction}
	
	In this section, we study the topology of $M$.
	
	\begin{proposition}\label{prop:topological-reduction}
		Either $q\equiv0$, or, after passing to the orientable double cover if
		necessary, $M$ has one end and
		\begin{equation*}
			H_2(M;\Z)=0.
		\end{equation*}
	\end{proposition}
	
	\begin{proof}
		We follow Corollary~3.11 of
		Catino--\hspace{0pt}Mari--\hspace{0pt}Mastrolia--\hspace{0pt}Roncoroni \cite{CMMR}.
		First observe that all analytic information lifts to a covering
		$\pi:\widehat M\to M$: the function $\widehat u=u\circ\pi$ satisfies the
		lifted Jacobi equation, while stability and the intrinsic curvature
		inequalities lift to $\widehat M$.
		
		\smallskip
		\noindent\emph{Case 1: $M$ has at least two ends.}
		The spectral condition \eqref{eq:spectral-ricci-nonnegative} has parameter
		$2/3<2$.
		The sharp spectral splitting theorem of Antonelli--Pozzetta--Xu \cite{APX}
		gives
		\begin{equation*}
			M\cong N^2\times\R,
		\end{equation*}
		with $N$ compact and $\Ric_N\geq0$.
		Hence $M$ has linear volume growth and is parabolic.
		For parabolic cutoffs $\eta_j\to1$ locally,
		\begin{equation*}
			\int_M|\nabla\eta_j|^2\longrightarrow0.
		\end{equation*}
		Using $\eta_j$ in \eqref{eq:stability} yields $q\equiv0$.
		
		\smallskip
		\noindent\emph{Case 2: $M$ has one end, but a connected twofold normal
			cover $\widehat M$ has two ends.}
		Applying the preceding splitting argument upstairs gives
		\begin{equation*}
			\widehat M=P'\times\R,
		\end{equation*}
		where $P'$ is compact.
		Since the quotient $M$ has only one end, the deck involution must exchange
		the two ends of the product.
		It therefore has the form
		\begin{equation*}
			\tau(t,y)=(-t+a,f(y)),
		\end{equation*}
		for an isometric involution $f$ of $P'$.
		Thus the quotient is precisely the determinant-line bundle described by
		Catino--Mari--Mastrolia--Roncoroni \cite[Corollary~3.11]{CMMR}; locally its
		metric is $dt^2+g_{P'}$ in the fibre direction.
		The product cover is parabolic.
		Indeed, if $\chi_R(t)=1$ for $|t|\leq R$, $\chi_R(t)=0$ for
		$|t|\geq2R$, and $|\chi_R'|\leq2/R$, then
		\begin{equation*}
			\int_{\widehat M}|\nabla\chi_R|^2
			\leq\frac{C\Area(P')}{R}\longrightarrow0.
		\end{equation*}
		Stability on $\widehat M$ consequently gives
		\begin{equation*}
			\int_{\widehat M}\widehat q\,\chi_R^2
			\leq\int_{\widehat M}|\nabla\chi_R|^2\longrightarrow0.
		\end{equation*}
		Thus $\widehat q\equiv0$, and the local-isometry property of the covering
		implies $q\equiv0$ on $M$.
		
		\smallskip
		\noindent\emph{Case 3: $M$ and every connected twofold normal cover are
			one-ended.}
		The remaining topological alternative in
		\cite[Corollary~3.11]{CMMR} gives
		\begin{equation*}
			H_c^1(M)=0.
		\end{equation*}
		Since $M$ is now oriented, the integral codimension-one conclusion in the
		same corollary, equivalently Poincar\'e duality, yields
		\begin{equation*}
			H_2(M;\Z)=0.
		\end{equation*}
		This is exactly the nontrivial branch in the statement.
	\end{proof}
	
	Henceforth we work in the nontrivial case
	\begin{equation}\label{eq:working-topology}
		M\text{ is orientable and one-ended},
		\qquad H_2(M;\Z)=0.
	\end{equation}
	One consequence of \eqref{eq:working-topology} used below is that every
	closed oriented surface in $M$ separates.
	This is enough to apply $\mu$-bubble theory later.
	
	The following proposition uses the two spectral inequalities satisfied by
	the same Jacobi function $u$ to give a stronger conclusion of independent
	interest.
	
	\begin{proposition}\label{prop:R3-refinement}
		Assume that $q\not\equiv0$, and make the orientable-cover reduction in
		\Cref{prop:topological-reduction}.
		Then the resulting orientable
		three-manifold is diffeomorphic to $\R^3$.
	\end{proposition}
	
	\begin{proof}
		The nontrivial case is noncompact, since on a compact stable
		hypersurface the test function $1$ and $q\geq0$ would give $q\equiv0$.
		Let $u>0$ be the Jacobi function.
		By
		\eqref{eq:rho-control} and $-\Delta u=qu$,
		\begin{equation*}
			\begin{split}
				-\Delta u+\Ric_Mu
				&=\bigl(q+\Ric_M\bigr)u\geq\frac13qu\geq0.
			\end{split}
		\end{equation*}
		The same function satisfies
		\begin{equation}\label{eq:liu-spectral-scalar}
			-\Delta u+\frac12R_Mu
			\geq\frac\kappa2u>0
		\end{equation}
		by \eqref{eq:spectral-biric-positive}.
		Thus the two hypotheses in the
		spectral Liu-type theorem of Chai--Sun
		\cite[Theorem~1.25 and Remark~1.26]{ChaiSun} hold for the same $u$, with
		the common parameter $\gamma=1<2$.
		It follows that either $M\cong\R^3$, or its universal Riemannian cover splits isometrically as
		\begin{equation*}
			\widetilde M=P^2\times\R,
			\qquad K_P\geq0,
		\end{equation*}
		where $P$ is complete.
		We show that the splitting alternative forces
		$q\equiv0$.
		
		All the analytic inequalities lift to $\widetilde M$.
		If $P$ is compact, then $P\times\R$ is parabolic.
		Taking cutoffs depending only on the
		$\R$-coordinate in the lifted stability inequality gives
		$\widetilde q\equiv0$,
		and hence $q\equiv0$, a contradiction.
		
		Suppose instead that $P$ is noncompact.
		Since the universal cover is simply connected, $P$ is simply connected.
		By the Cohn--Vossen total-curvature theorem \cite{CohnVossen} and the
		Calabi--Yau volume theorem \cite{Yau76},
		\begin{equation*}
			\int_P K_P\leq2\pi,
			\qquad
			\Area_P(B_R)\longrightarrow\infty.
		\end{equation*}
		Moreover, Bishop--Gromov gives $\Area_P(B_R)\leq CR^2$.
		Choose a distance cutoff $\zeta_R$ on $P$ which equals $1$ on $B_R$,
		vanishes outside $B_{2R}$, and satisfies $|\nabla\zeta_R|\leq2/R$.
		Then
		\begin{equation*}
			\frac{\displaystyle
				\int_P\bigl(|\nabla\zeta_R|^2+K_P\zeta_R^2\bigr)}
			{\displaystyle\int_P\zeta_R^2}
			\longrightarrow0.
		\end{equation*}
		Combining $\zeta_R$ with increasingly long one-dimensional cutoffs on the
		line factor shows
		\begin{equation}\label{eq:product-bottom-zero}
			\lambda_1\left(-\Delta_{P\times\R}+K_P\right)=0.
		\end{equation}
		On the other hand, the lift of \eqref{eq:liu-spectral-scalar}, together
		with $R_{P\times\R}=2K_P$, supplies a positive supersolution of
		\begin{equation*}
			-\Delta_{P\times\R}+K_P\geq\frac\kappa2.
		\end{equation*}
		Therefore,
		\begin{equation*}
			\lambda_1\left(-\Delta_{P\times\R}+K_P\right)
			\geq\frac\kappa2,
		\end{equation*}
		contradicting \eqref{eq:product-bottom-zero}.
		Thus the splitting
		case is impossible when $q\not\equiv0$, and the spectral Liu-type
		theorem yields $M\cong\R^3$.
	\end{proof}

	\section{Separators from warped \texorpdfstring{$\mu$}{mu}-bubbles}
	\label{sec:bubble-separators}
	
	We now construct the outer boundaries for the harmonic function identity.
	The construction is the $n=3$, $\gamma=1$ case of the warped $\mu$-bubble
	method; see, for example, Antonelli--Xu \cite{AX}.
	We retain the terms involving the normal derivative of $u$ in the second
	variation formula for the $\mu$-bubble functional.
	
	In this section, $D$ denotes the Levi--Civita connection of the ambient
	three-manifold $M$, whereas $\nabla$ denotes the induced Levi--Civita
	connection on the surface $\Sigma$.
	We use $\Delta_M$ and
	$\Delta_\Sigma$ for their respective Laplacians.
	
	Let $K\subset M$ be a large connected compact set.
	Choose smooth precompact domains
	\begin{equation*}
		K\Subset\Omega_-\Subset\Omega_+\Subset M,
		\qquad
		N:=\overline{\Omega_+\setminus\Omega_-},
	\end{equation*}
	and fix a smooth reference domain $\Omega_0$ with
	$\Omega_-\Subset\Omega_0\Subset\Omega_+$.
	Let $h\in C^\infty(\operatorname{int}N)$ tend uniformly to $+\infty$ at
	$\partial\Omega_-$ and to $-\infty$ at $\partial\Omega_+$.
	For every Caccioppoli set $\Omega$ satisfying
	$\Omega\mathbin{\triangle}\Omega_0\Subset\operatorname{int}N$, define
	\begin{equation}\label{eq:bubble-energy}
		\mathcal E(\Omega):=\int_{\partial_N^*\Omega}u\,\dd A
		-\int_N(\chi_\Omega-\chi_{\Omega_0})hu\,\dd V,
		\qquad
		\partial_N^*\Omega:=\partial^*\Omega\cap\operatorname{int}N.
	\end{equation}
	The direct method, together with the two barrier conditions on $h$, produces
	a minimizer $\Omega$ whose relative reduced boundary stays a positive distance
	from $\partial N$; see \cite{Zhu-width-estimate}.
	In dimension three, standard regularity makes this boundary a smooth closed
	surface, possibly with several connected components.
	Let $\Sigma$ be a connected component and let $\nu$ be its outward normal.
	
	The first variation of $\mathcal E$ at $\Omega$ gives the critical equation
	\begin{equation}\label{eq:bubble-EL}
		H+\partial_\nu\log u=h.
	\end{equation}
	
	\begin{lemma}\label{lem:bubble-estimate}
		Fix $\alpha\in(0,1/2)$.
		The function $h$ may be chosen so that
		\begin{equation}\label{eq:h-gradient}
			|Dh|\leq \alpha h^2+\frac\kappa4.
		\end{equation}
		On the connected separating component $\Sigma$,
		\begin{equation}\label{eq:coercive-separator}
			\begin{split}
				\frac\kappa4\Area(\Sigma)
				&+\left(\frac12-\alpha\right)\int_\Sigma h^2
				+\frac12\int_\Sigma|\partial_\nu\log u|^2
				+\frac34\int_\Sigma|\nabla\log u|^2
				\leq4\pi.
			\end{split}
		\end{equation}
		Consequently,
		\begin{equation}\label{eq:H-bound}
			\int_\Sigma
			\left(H^2+h^2+|\partial_\nu\log u|^2\right)
			\leq C(\alpha).
		\end{equation}
	\end{lemma}
	
	\begin{proof}
		Let $A$ denote the second fundamental form of $\Sigma\subset M$.
		For a normal variation with speed $\varphi$, the second variation of
		$\mathcal E$ at $\Omega$ is
		\begin{equation*}
			\begin{split}
				0\leq Q(\varphi):={}&
				\int_\Sigma u\Bigl\{ |\nabla\varphi|^2
				-\bigl[\Ric_M(\nu,\nu)+|A|^2
				-D^2\log u(\nu,\nu)+\partial_\nu h\bigr]
				\varphi^2\Bigr\}.
			\end{split}
		\end{equation*}
		Let $\varphi=u^{-1/2}$.
		Since $u\varphi^2=1$ and
		$u|\nabla\varphi|^2=\frac14|\nabla\log u|^2$, we obtain
		\begin{equation*}
			\begin{split}
				0\leq Q(u^{-1/2})
				=\int_\Sigma\Bigl\{
				&\frac14|\nabla\log u|^2
				-\Ric_M(\nu,\nu)-|A|^2\\
				&+D^2\log u(\nu,\nu)-\partial_\nu h
				\Bigr\}.
			\end{split}
		\end{equation*}
		The tangential-normal decomposition of the Laplacian and the Jacobi equation
		give
		\begin{equation*}
			\begin{split}
				D^2\log u(\nu,\nu)
				&={}\Delta_M\log u-\Delta_\Sigma\log u
				-H\partial_\nu\log u,\\
				\Delta_M\log u
				&={}u^{-1}\Delta_M u-|D\log u|^2
				=-q-|D\log u|^2.
			\end{split}
		\end{equation*}
		Because $\Sigma$ is closed, the integral of
		$\Delta_\Sigma\log u$ vanishes.
		Hence
		\begin{equation*}
			\begin{split}
				0\leq \int_\Sigma\Bigl[{}
				&-\Ric_M(\nu,\nu)-|A|^2-q
				-\frac34|\nabla\log u|^2\\
				&-|\partial_\nu\log u|^2
				-H\partial_\nu\log u-\partial_\nu h
				\Bigr].
			\end{split}
		\end{equation*}
		The Gauss equation for $\Sigma\subset M$ is
		\begin{equation*}
			-\Ric_M(\nu,\nu)-|A|^2
			=K_\Sigma-\frac12R_M-\frac12(H^2+|A|^2).
		\end{equation*}
		Since $|A|^2\geq H^2/2$ in dimension two,
		$R_M/2+q\geq\kappa/2$ by \eqref{eq:scalar-gauss}, and
		$-\partial_\nu h\leq|Dh|$, the preceding identity implies
		\begin{equation*}
			\begin{split}
				0\leq \int_\Sigma\Bigl[{}
				&-\frac\kappa2+K_\Sigma+|Dh|
				-\frac34H^2-|\partial_\nu\log u|^2\\
				&-H\partial_\nu\log u
				-\frac34|\nabla\log u|^2
				\Bigr].
			\end{split}
		\end{equation*}
		Finally, substituting $H=h-\partial_\nu\log u$ from
		\eqref{eq:bubble-EL} gives
		\begin{equation*}
			\begin{split}
				&-\frac34H^2-|\partial_\nu\log u|^2
				-H\partial_\nu\log u\\
				&\hspace{2cm}={}-\frac34h^2
				+\frac12h\,\partial_\nu\log u
				-\frac34|\partial_\nu\log u|^2\\
				&\hspace{2cm}\le{}-\frac12h^2
				-\frac12|\partial_\nu\log u|^2,
			\end{split}
		\end{equation*}
		where we have used the Cauchy--Schwarz inequality.
		
		Combining these calculations yields
		\begin{equation}\label{eq:separator-integral}
			0\leq
			\int_\Sigma
			\left[
			-\frac\kappa2+K_\Sigma+|Dh|
			-\frac12h^2
			-\frac12|\partial_\nu\log u|^2
			-\frac34|\nabla\log u|^2
			\right].
		\end{equation}
		
		Using \eqref{eq:h-gradient}, \eqref{eq:separator-integral}, and the Gauss--Bonnet
		theorem,
		\begin{equation*}
			\int_\Sigma K_\Sigma=2\pi\chi(\Sigma)\leq4\pi,
		\end{equation*}
		we obtain \eqref{eq:coercive-separator}.
		This estimate first gives
		\begin{equation*}
			\int_\Sigma h^2
			\leq\frac{4\pi}{\frac12-\alpha},
			\qquad
			\int_\Sigma|\partial_\nu\log u|^2\leq8\pi.
		\end{equation*}
		On the other hand, \eqref{eq:bubble-EL} gives
		$H=h-\partial_\nu\log u$, and hence
		\begin{equation*}
			H^2=\left(h-\partial_\nu\log u\right)^2
			\leq2\left(h^2+|\partial_\nu\log u|^2\right)
		\end{equation*}
		pointwise on $\Sigma$.
		Consequently,
		\begin{equation*}
			\begin{split}
				\int_\Sigma
				\left(H^2+h^2+|\partial_\nu\log u|^2\right)
				&\leq3\int_\Sigma
				\left(h^2+|\partial_\nu\log u|^2\right)\leq\frac{12\pi}{\frac12-\alpha}+24\pi,
			\end{split}
		\end{equation*}
		which proves \eqref{eq:H-bound}.
		
		To obtain \eqref{eq:h-gradient}, take a smooth $2$-Lipschitz approximation
		$s$ of the distance from $K$ and put $h=\lambda\cot(\sigma s)$ on a fixed
		finite annulus.
		Choosing $\sigma/\lambda$ sufficiently small gives
		$|Dh|\leq\alpha h^2+\kappa/4$ without any curvature estimate.
		The poles of the cotangent provide the barrier at the two faces.
	\end{proof}
	
	Applying the construction with $K\supset B_R(o)$ produces connected
	surfaces $\Sigma_R$ escaping to infinity and satisfying
	\eqref{eq:coercive-separator} uniformly in $R$.
	The corresponding minimizer of $\mathcal E$ can be filled across all
	bounded complementary components.
	Since $M$ has one end and $H_2(M;\Z)=0$, the resulting outer boundary is
	connected.
	
	\section{The harmonic function identity}\label{sec:harmonic-identity}
	
	Throughout this section, $D$ denotes the Levi--Civita connection and gradient
	on $M$, while $\nabla$ denotes the intrinsic connection and gradient on a
	regular level surface of the capacitor.
	Thus
	$\Delta=\operatorname{tr}_g D^2$.
	
	Fix a connected smooth inner boundary $\Gamma$ enclosing a compact set.
	Let $D_R$ be the compact region bounded by $\Gamma$ and $\Sigma_R$, and
	let $v=v_R$ solve
	\begin{equation}\label{eq:capacitor}
		\Delta v=0\quad\text{in }D_R,
		\qquad
		v=1\quad\text{on }\Gamma,
		\qquad
		v=0\quad\text{on }\Sigma_R.
	\end{equation}
	
	\begin{lemma}\label{lem:connected-levels}
		Every regular level
		\begin{equation*}
			L_t:=\{x\in D_R:v_R(x)=t\},
			\qquad 0<t<1,
		\end{equation*}
		is a connected closed surface.
	\end{lemma}
	
	\begin{proof}
		Let $C$ be a connected component of $L_t$.
		Regularity makes $C$ a smooth closed two-sided surface.
		Since $M$ is oriented and
		$H_2(M;\Z)=0$, the surface $C$ separates $M$.
		
		We first claim that $C$ separates the two boundary components $\Gamma$
		and $\Sigma_R$.
		Otherwise they lie on the same side of $C$, and the
		other side cuts out a compact subdomain $\Omega\Subset D_R$ whose boundary
		is contained in $C$.
		Since $v_R=t$ on $\partial\Omega$, the maximum and minimum principles give
		$v_R\equiv t$ on $\Omega$.
		This contradicts both
		the unique continuation principle and the boundary values in
		\eqref{eq:capacitor}.
		
		Suppose now that $L_t$ has two distinct components $C_1$ and $C_2$.
		Both separate $\Gamma$ from $\Sigma_R$.
		Because they are disjoint and
		connected, they are ordered between the two boundary components: one
		lies on the $\Gamma$-side of the other.
		The closure of the region between
		them is compact, and every component of its boundary is contained in
		$L_t$.
		Thus $v_R=t$ on the entire boundary of this region.
		Applying the
		maximum and minimum principles once more forces $v_R\equiv t$ there,
		again a contradiction.
		Hence $L_t$ has only one component.
	\end{proof}
	
	The following is the harmonic specialization of Munteanu--Wang
	\cite[Lemma~2.9 and (2.19)]{MW24}; see also
	\cite[Theorem~4.1]{MW}.
	All calculations are performed on the regular set of $v$.
	\begin{lemma}\label{lem:bochner-gauss}
		On a regular level of $v$,
		\begin{equation}\label{eq:bochner-gauss}
			\Delta|Dv|
			\geq
			\frac12(R_M-R_t)|Dv|
			+\frac34\frac{|D|Dv||^2}{|Dv|},
		\end{equation}
		where $R_t=2K_{\{v=t\}}$ is the scalar curvature of the level surface.
		
	\end{lemma}
	
	\begin{proof}
		We include the calculation to fix our curvature conventions and to keep
		the nonnegative remainder terms that will be useful below.
		Let $n=Dv/|Dv|$, and let $A$ and $H$ be the second fundamental form and mean
		curvature of the level surface.
		Choose an orthonormal frame $e_1,e_2$ tangent to the level surface.
		Since
		$Dv=|Dv|n$, the components of the Hessian are
		\begin{equation*}
			\begin{split}
				D^2v(e_i,e_j)&=|Dv|A_{ij},\\
				D^2v(e_i,n)&=e_i(|Dv|),\\
				D^2v(n,n)&=D_n|Dv|.
			\end{split}
		\end{equation*}
		Taking the trace and using $\Delta v=0$ gives
		\begin{equation*}
			D_n|Dv|+H|Dv|=0.
		\end{equation*}
		Consequently,
		\begin{equation*}
			|D^2v|^2
			=|Dv|^2|A|^2
			+2|\nabla|Dv||^2
			+H^2|Dv|^2.
		\end{equation*}
		
		The Gauss equation and the trace decomposition of $A$ are
		\begin{equation*}
			\begin{split}
				R_t&=R_M-2\Ric_M(n,n)+H^2-|A|^2,\\
				|A|^2&=|A^\circ|^2+\frac12H^2.
			\end{split}
		\end{equation*}
		The Bochner formula gives
		\begin{equation}
			\frac{\Delta|Dv|}{|Dv|}
			=\frac12(R_M-R_t)
			+\frac12|A^\circ|^2
			+\frac34H^2
			+|\nabla\log|Dv||^2.
		\end{equation}
		Since
		$|D\log|Dv||^2
		=H^2+|\nabla\log|Dv||^2$, we obtain
		\eqref{eq:bochner-gauss}.
	\end{proof}
	
	The Jacobi equation can be written
	\begin{equation}\label{eq:q-logu}
		q=-\Delta\log u-|D\log u|^2.
	\end{equation}
	
	\begin{proposition}\label{prop:capacitor-inequality}
		There is a constant $C$, independent of $R$, such that
		\begin{equation}\label{eq:L1-capacitor}
			\int_{D_R}|Dv_R|\leq C.
		\end{equation}
	\end{proposition}
	
	\begin{proof}
		For $0\leq t\leq1$, define
		\begin{equation}\label{eq:theta-ode}
			\theta(t):=(2-t)^{8/11},
			\qquad
			-\theta''(t)-\frac{3\theta'^2(t)}{8\theta(t)}=0.
		\end{equation}
		We write $\theta$ for the composition $\theta(v)$ throughout the proof.
		Notice also that $\theta'<0$ and $1\leq\theta\leq2^{8/11}$.
		
		Consider the Green identity
		\begin{equation}\label{eq:I-def}
			I_R:=\int_{D_R}\left(
			\theta\Delta|Dv|-|Dv|\Delta\theta\right)
			=\int_{\partial D_R}
			\left(\theta D_\nu|Dv|
			-|Dv|D_\nu\theta\right).
		\end{equation}
		Since $v$ is harmonic, $\Delta\theta=\theta''|Dv|^2$.
		Moreover, \eqref{eq:scalar-gauss} and \eqref{eq:q-logu} give
		$R_M\geq\kappa-2q$.
		Substituting these relations and \eqref{eq:bochner-gauss} into
		\eqref{eq:I-def} gives
		\begin{equation*}
			\begin{split}
				I_R\geq{}&
				\frac\kappa2\int_{D_R}\theta|Dv|
				-\frac12\int_{D_R}\theta R_t|Dv|
				-\int_{D_R}q\theta|Dv|\\
				&+\frac34\int_{D_R}\theta
				\frac{|D|Dv||^2}{|Dv|}
				-\int_{D_R}\theta''|Dv|^3.
			\end{split}
		\end{equation*}

		It remains to rewrite the third term.
		Using \eqref{eq:q-logu} and integrating the Laplacian term by parts, we obtain
		\begin{equation}\label{eq:q-integration}
			\begin{split}
				-\int_{D_R}q\theta|Dv|
				={}&\int_{\partial D_R}\theta|Dv|
				\,D_\nu\log u-\int_{D_R}\left\langle
				D\bigl(\theta|Dv|\bigr),
				D\log u\right\rangle\\
				&+\int_{D_R}\theta|Dv|
				|D\log u|^2.
			\end{split}
		\end{equation}
		Substituting \eqref{eq:q-integration} into the preceding inequality, we
		obtain
		\begin{equation*}
			\begin{split}
				I_R\geq{}&
				\frac\kappa2\int_{D_R}\theta|Dv|
				-\frac12\int_{D_R}\theta R_t|Dv|
				+\int_{\partial D_R}\theta|Dv|D_\nu\log u\\
				&+\int_{D_R}\left\{
				\frac34\theta\frac{|D|Dv||^2}{|Dv|}
				-\theta''|Dv|^3\right.\\
				&\hspace{3.7cm}\left.
				-\left\langle D(\theta|Dv|),D\log u\right\rangle
				+\theta|Dv||D\log u|^2\right\}.
			\end{split}
		\end{equation*}
		We now estimate the integrand in the last integral.
		Since
		\begin{equation*}
			\begin{split}
				D(\theta|Dv|)
				&=\theta D|Dv|+\theta'|Dv|Dv,\\
				D|Dv|&=|Dv|D\log|Dv|
			\end{split}
		\end{equation*}
		on the regular set, we obtain
		\begin{equation}
			\begin{split}
				\widetilde P={}&\theta|Dv|
				\left(
				\frac34|D\log|Dv||^2
				-\left\langle D\log|Dv|,D\log u\right\rangle
				+|D\log u|^2\right)\\
				&-\theta'|Dv|
				\left\langle Dv,D\log u\right\rangle
				-\theta''|Dv|^3\\
				\ge{}&\theta|Dv|
				\left(
				\frac23|D\log u|^2\right)-\theta'|Dv|
				\left\langle Dv,D\log u\right\rangle
				-\theta''|Dv|^3\\
				\ge{}&\frac23\theta|Dv|(D_n\log u)^2
				-\theta'|Dv|^2D_n\log u
				-\theta''|Dv|^3\\
				\geq{}&
				-\frac{3(\theta')^{2}}{8\theta}|Dv|^3-\theta''|Dv|^3=0.
			\end{split}
		\end{equation}
		Here, we have used the Cauchy--Schwarz inequality in the first and third inequalities to control the mixed terms, and the second inequality follows from the fact that $|D\log u|^2 \geq (D_n\log u)^2$.
		
		We conclude that
		\begin{equation}\label{eq:pre-boundary}
			\begin{split}
				\frac\kappa2\int_{D_R}\theta|Dv|
				\leq{}&
				\frac12\int_{D_R}\theta R_t|Dv|\\
				&+\int_{\partial D_R}
				\left(\theta D_\nu|Dv|
				-|Dv|D_\nu\theta
				-\theta|Dv|D_\nu\log u\right).
			\end{split}
		\end{equation}
		
		Every regular level of $v$ is connected by \Cref{lem:connected-levels}.
		Therefore, the coarea and Gauss--Bonnet formulas give
		\begin{equation}
			\begin{split}
				\frac12\int_{D_R}\theta R_t|Dv|
				&=\frac12\int_0^1\theta(t)
				\left(\int_{\{v=t\}}R_t\right)\dd t\leq4\pi\int_0^1\theta(t)\,\dd t.
			\end{split}
		\end{equation}
		
		On the outer boundary, let $\nu$ point out of $D_R$, toward infinity.
		Since $v_R=0$ there and $v_R$ decreases toward infinity,
		\begin{equation*}
			D_\nu v_R=-|Dv_R|,
			\qquad
			D_\nu|Dv_R|=-H|Dv_R|,
			\qquad
			D_\nu\theta=-\theta'|Dv_R|.
		\end{equation*}
		Thus by the $\mu$-bubble critical equation
		$H+D_\nu\log u=h$, the boundary integrand
		in \eqref{eq:pre-boundary} is exactly
		\begin{equation}\label{eq:outer-boundary}
			\begin{split}
				-\theta h|Dv_R|+\theta'|Dv_R|^2
				&\leq\frac{\theta^2}{4|\theta'|}h^2.
			\end{split}
		\end{equation}
		Here we used $\theta'<0$ and the Cauchy--Schwarz inequality.
		Since $v_R=0$ on $\Sigma_R$, both $\theta(0)$ and $|\theta'(0)|$ are fixed
		positive constants.
		Hence
		\begin{equation*}
			\int_{\Sigma_R}
			\left(-\theta h|Dv_R|+\theta'|Dv_R|^2\right)
			\leq
			\frac{\theta^2(0)}{4|\theta'(0)|}
			\int_{\Sigma_R}h^2
			\leq C,
		\end{equation*}
		where the last bound is uniform in $R$ by \Cref{lem:bubble-estimate}.
		The inner boundary of $\Omega_R$ is fixed.
		As $1\leq\theta\leq2^{8/11}$, equations
		\eqref{eq:pre-boundary}--\eqref{eq:outer-boundary} prove
		\eqref{eq:L1-capacitor}.
	\end{proof}
	
	\begin{remark}\label{rem:critical-set}
		The calculation can be justified globally by replacing $|Dv|$ by
		$(|Dv|^2+\varepsilon)^{1/2}$,
		performing all integrations on the fixed compact domain $D_R$, and then
		letting $\varepsilon\downarrow0$.
		Standard estimates for the critical set of a harmonic function control the
		error terms; this is the same regularization used in Green-level proofs such
		as Munteanu--Wang \cite{MW}.
		The $\varepsilon$ limit is taken before $R\to\infty$, so no curvature bound
		uniform in $R$ is required.
	\end{remark}
	
	\section{Proof of the main theorem}\label{sec:main-proof}
	
	\begin{proof}[Proof of \Cref{thm:main}]
		By \Cref{prop:topological-reduction}, it remains to consider the situation in
		\eqref{eq:working-topology}.
		Suppose for contradiction that $M$ is nonparabolic.
		A subsequence of the capacitors $v_R$ converges smoothly on compact subsets
		of the exterior of $\Gamma$ to a nonconstant function $v$.
		Fatou's lemma and \Cref{prop:capacitor-inequality} give
		\begin{equation}\label{eq:finite-L1-limit}
			\int_{M\setminus K}|\nabla v|<\infty.
		\end{equation}
		
		Let $K$ be the compact region bounded by $\Gamma$, and let $\nu_K$ point
		from $K$ into $M\setminus K$.
		Since $v=1$ on $\Gamma$ and $v$ is
		nonconstant, the Hopf boundary lemma gives
		\begin{equation*}
			c:=-\int_\Gamma D_{\nu_K}v>0.
		\end{equation*}
		For almost every sufficiently large $r$, the divergence theorem applied
		to $B_r\setminus K$ gives
		\begin{equation*}
			\begin{split}
				0
				&=\int_{B_r\setminus K}\Delta v=\int_{\partial B_r}D_{\nu_r}v
				-\int_\Gamma D_{\nu_K}v,
			\end{split}
		\end{equation*}
		where $\nu_r$ is the outward normal to $\partial B_r$.
		Thus
		$\int_{\partial B_r}D_{\nu_r}v=-c$, and consequently
		\begin{equation}
			\begin{split}
				\int_{\partial B_r}|Dv|
				&\geq\left|\int_{\partial B_r}D_{\nu_r}v\right|
				=c.
			\end{split}
		\end{equation}
		Coarea now gives
		\begin{equation}
			\int_{B_T\setminus B_{T_0}}|\nabla v|
			\geq c(T-T_0),
		\end{equation}
		contradicting \eqref{eq:finite-L1-limit}.
		Thus $M$ is parabolic.
		
		Choose parabolic cutoffs $\eta_j\to1$ locally, with
		$\int|\nabla\eta_j|^2\to0$.
		Stability yields
		\begin{equation}
			\int_M q\eta_j^2\leq\int_M|\nabla\eta_j|^2\longrightarrow0.
		\end{equation}
		Since $q\geq0$, it follows that $q\equiv0$.
		Finally,
		\begin{equation*}
			q=|A|^2+\sum_{i=1}^3\overline K(\nu,e_i)
		\end{equation*}
		is a sum of nonnegative functions.
		Therefore $A\equiv0$ and
		$\overline\Ric(\nu,\nu)\equiv0$.
	\end{proof}
	
	
	\section{The counterexample construction}\label{sec:counterexample-proof}
	
	We now give the complete proof of \Cref{thm:counterexample}.
	The main point is that the first and second normal derivatives of the ambient
	metric can be changed simultaneously so that the hypersurface has a nonzero
	trace-free second fundamental form while its Jacobi potential remains
	unchanged.
	
	\begin{proposition}\label{prop:normal-derivative-deformation}
		Let $\Sigma^3$ be an embedded, two-sided, totally geodesic hypersurface in
		a Riemannian four-manifold $(X^4,G_0)$, and denote its induced metric by
		$g:=G_0|_{T\Sigma}$.
		Let $K\Subset\Sigma$, and suppose
		that $G_0$ has strictly positive sectional curvature on a neighborhood of
		$K$.
		Given a nonzero tensor
		\begin{equation*}
			T\in C_c^\infty(\operatorname{Sym}^2T^*\Sigma),
			\qquad
			\operatorname{supp}T\subset K,
			\qquad
			\tr_gT=0,
		\end{equation*}
		there are, for all sufficiently small $|\varepsilon|>0$, smooth metrics
		$G_\varepsilon$ on $X$ such that:
		\begin{enumerate}[label=\textup{(\roman*)}]
			\item $G_\varepsilon=G_0$ outside a compact neighborhood of $K$;
			\item $G_\varepsilon|_{T\Sigma}=G_0|_{T\Sigma}=g$;
			\item $\Sigma$ is minimal in $(X,G_\varepsilon)$ and its second
			fundamental form is $A_\varepsilon=\varepsilon T\not\equiv0$;
			\item the Jacobi potential satisfies
			\begin{equation*}
				|A_\varepsilon|^2
				+\Ric_{G_\varepsilon}(\nu,\nu)
				=\Ric_{G_0}(\nu,\nu);
			\end{equation*}
			\item $\Sec_{G_\varepsilon}>0$ on the perturbation region.
		\end{enumerate}
		In particular, if $\Sigma$ is stable for $G_0$, it remains stable for
		$G_\varepsilon$ but is no longer totally geodesic.
	\end{proposition}
	
	\begin{proof}
		Choose a relatively compact open set $U\Subset\Sigma$ containing $K$.
		Since $K$ is compact and $\Sigma$ is embedded and two-sided, the normal
		exponential map gives, after shrinking $U$ if necessary, a diffeomorphism
		\begin{equation*}
			\Phi:U\times(-\delta,\delta)\longrightarrow X,
			\qquad
			\Phi(x,t)=\exp_x(t\nu_x),
		\end{equation*}
		onto its image for some $\delta>0$.
		These are Fermi coordinates only around the compact set $K$.
		The Gauss lemma gives
		\begin{equation*}
			G_0=dt^2+g_t^0,
			\qquad
			\Sigma\cap\Phi(U\times(-\delta,\delta))=\{t=0\},
			\qquad
			\nu=\partial_t.
		\end{equation*}
		
		We first calculate the required normal-coordinate identities.
		For a metric $G=dt^2+g_t$, write $g:=g_{t=0}$ for the induced metric on the
		central slice.
		With the convention
		\begin{equation*}
			A_{ij}=G(D_{\partial_i}\nu,\partial_j),
		\end{equation*}
		one has
		\begin{equation*}
			A_{ij}=\frac12\dot g_{ij}.
		\end{equation*}
		For every $t$ for which the Fermi coordinates are defined, the normal
		curvature satisfies
		\begin{equation*}
			R_{\nu i\nu j}
			=-\frac12\ddot g_{ij}+(A^2)_{ij}.
		\end{equation*}
		To see this, use
		\begin{equation*}
			D_\nu\nu=0,
			\qquad
			D_{\partial_i}\nu=D_\nu\partial_i=A_i{}^k\partial_k,
			\qquad
			G(D_{\partial_i}\partial_j,\nu)=-A_{ij}.
		\end{equation*}
		With the curvature convention
		$
		R(X,Y)Z=D_XD_YZ-D_YD_XZ-D_{[X,Y]}Z,
		$
		we obtain
		\begin{align*}
			R_{\nu i\nu j}
			&=G(D_\nu D_{\partial_i}\partial_j,\nu)
			-G(D_{\partial_i}D_\nu\partial_j,\nu)\\
			&=-\dot A_{ij}+A_j{}^kA_{ik}
			=-\frac12\ddot g_{ij}+(A^2)_{ij}.
		\end{align*}
		Taking the trace with respect to the metric $g_t$ on the slice
		$\{t=\mathrm{constant}\}$ gives, throughout the collar,
		\begin{equation*}
			\Ric_G(\nu,\nu)
			=-\frac12\tr_{g_t}\ddot g_t+|A_t|^2.
		\end{equation*}
		Specializing to $t=0$, where $g=g_{t=0}$, gives
		\begin{equation*}
			\Ric_G(\nu,\nu)\big|_{t=0}
			=-\frac12\tr_g\ddot g_0+|A_0|^2.
		\end{equation*}
		Hence the Jacobi potential of a minimal central slice is
		\begin{equation*}
			\left(\Ric_G(\nu,\nu)+|A|^2\right)\big|_{t=0}
			=-\frac12\tr_g\ddot g_0+2|A_0|^2.
		\end{equation*}
		
		Because the original slice is totally geodesic,
		$\dot g_0^0=0$.
		Choose an even cutoff
		$\chi\in C_c^\infty((-\delta,\delta))$ which is identically one near
		$t=0$.
		Extend $T$ and $g=g_0^0$ constantly in the $t$ direction and define
		\begin{equation*}
			g_t^\varepsilon
			:=g_t^0+2\varepsilon t\chi(t)T
			+\frac23\varepsilon^2t^2\chi(t)|T|_g^2g.
		\end{equation*}
		Define $G_\varepsilon=dt^2+g_t^\varepsilon$ on the collar and
		$G_\varepsilon=G_0$ outside it.
		The compact support of $T$ handles the lateral boundary, while $\chi$ handles
		the two normal boundary faces, so these definitions glue smoothly.
		For small $|\varepsilon|$, the tensors
		$g_t^\varepsilon$ remain positive definite.
		
		Since $\chi(0)=1$ and $\chi'(0)=0$, the metric and its first two normal
		derivatives at the central slice are
		\begin{equation*}
			g_0^\varepsilon=g_0^0=g,
			\qquad
			\dot g_0^\varepsilon=2\varepsilon T,
			\qquad
			\ddot g_0^\varepsilon
			=\ddot g_0^0+\frac43\varepsilon^2|T|_g^2g.
		\end{equation*}
		It follows that
		\begin{equation*}
			A_\varepsilon=\varepsilon T,
			\qquad
			H_\varepsilon=\varepsilon\tr_gT=0.
		\end{equation*}
		Thus $\Sigma$ remains minimal and is not totally geodesic.
		Since $\dim\Sigma=3$,
		\begin{equation*}
			\tr_g(\ddot g_0^\varepsilon-\ddot g_0^0)
			=4\varepsilon^2|T|_g^2.
		\end{equation*}
		Since the original slice is totally geodesic, $A_0=0$, and the normal
		Ricci formula for $G_0$ gives
		\begin{equation*}
			\Ric_{G_0}(\nu,\nu)
			=-\frac12\tr_g\ddot g_0^0.
		\end{equation*}
		Applying the same formula to $G_\varepsilon$ and using the preceding trace
		identity gives
		\begin{align*}
			\Ric_{G_\varepsilon}(\nu,\nu)
			&=-\frac12\tr_g\ddot g_0^\varepsilon
			+|A_\varepsilon|_g^2\\
			&=-\frac12\tr_g\ddot g_0^0
			-2\varepsilon^2|T|_g^2
			+\varepsilon^2|T|_g^2\\
			&=\Ric_{G_0}(\nu,\nu)-\varepsilon^2|T|_g^2.
		\end{align*}
		Because $|A_\varepsilon|^2=\varepsilon^2|T|_g^2$, the two changes cancel:
		\begin{equation*}
			|A_\varepsilon|^2
			+\Ric_{G_\varepsilon}(\nu,\nu)
			=\Ric_{G_0}(\nu,\nu).
		\end{equation*}
		The induced metric and volume form on $\Sigma$ are unchanged.
		Therefore,
		for every $\varphi\in C_c^\infty(\Sigma)$,
		\begin{align*}
			Q_\varepsilon(\varphi)
			&=\int_\Sigma\left(
			|\nabla\varphi|_g^2
			-\bigl[|A_\varepsilon|^2
			+\Ric_{G_\varepsilon}(\nu,\nu)\bigr]\varphi^2
			\right)\dd\mu_g\\
			&=\int_\Sigma\left(
			|\nabla\varphi|_g^2
			-\Ric_{G_0}(\nu,\nu)\varphi^2
			\right)\dd\mu_g
			=Q_0(\varphi).
		\end{align*}
		Hence stability is preserved exactly.
		
		Finally, on the fixed compact perturbation region,
		\begin{equation*}
			\|G_\varepsilon-G_0\|_{C^2}=O(|\varepsilon|).
		\end{equation*}
		Sectional curvature depends continuously on the metric, its first two
		derivatives, and the tangent two-plane.
		The Grassmann bundle over this compact region is compact, and $\Sec_{G_0}$
		has a positive minimum there.
		Thus
		$\Sec_{G_\varepsilon}>0$ on the perturbation region for all sufficiently
		small $|\varepsilon|$.
		This proves the proposition.
	\end{proof}
	
	We next recall the complete starting model.
	For
	$\alpha\in(0,1)$ sufficiently close to one, define
	\begin{equation*}
		\rho(r)
		:=\alpha r+(1-\alpha)\int_0^r e^{-s^2}\,ds
	\end{equation*}
	and equip $\R^4$ with the rotationally symmetric metric
	\begin{equation*}
		G_0=dr^2+\rho^2(r)g_{S^3}.
	\end{equation*}
	This is the model of Chodosh--Li--Stryker
	\cite[Example~1.2 and Appendix~B.1]{CLS}.
	Since
	\begin{equation*}
		\rho'(r)=\alpha+(1-\alpha)e^{-r^2},
		\qquad
		\rho''(r)=-2(1-\alpha)re^{-r^2},
	\end{equation*}
	its radial and tangential sectional curvatures satisfy
	\begin{equation*}
		-\frac{\rho''}{\rho}>0,
		\qquad
		\frac{1-(\rho')^2}{\rho^2}>0
	\end{equation*}
	for $r>0$.
	The expansion $\rho(r)=r+O(r^3)$ at the origin gives a smooth
	metric there, and the radial $dr^2$ term makes $G_0$ complete.
	
	An equatorial copy
	\begin{equation*}
		\Sigma=[0,\infty)\times S^2
		\subset[0,\infty)\times S^3
	\end{equation*}
	is an embedded, two-sided, totally geodesic hypersurface diffeomorphic to
	$\R^3$, with induced metric
	\begin{equation*}
		g=dr^2+\rho^2(r)g_{S^2}.
	\end{equation*}
	For $\alpha$ sufficiently close to one, Chodosh--Li--Stryker \cite{CLS}
	use the Euclidean Hardy inequality and $\alpha r\leq\rho(r)\leq r$ to prove
	\begin{equation*}
		\int_\Sigma|\nabla\varphi|^2\dd\mu_g
		\geq
		\int_\Sigma\Ric_{G_0}(\nu,\nu)\varphi^2\dd\mu_g
	\end{equation*}
	for every $\varphi\in C_c^\infty(\Sigma)$.
	Hence the totally geodesic slice is stable.
	
	For completeness, we verify directly that $\Sigma$ is nonparabolic; this
	uses no Ricci-curvature assumption.
	For a radial function $f=f(r)$,
	\begin{equation*}
		\Delta_gf=f''+2\frac{\rho'}{\rho}f'.
	\end{equation*}
	The Dirichlet Green function of a model ball $B_R(o)$, with pole at the
	origin, is
	\begin{equation*}
		G_R(o,x)=\frac1{4\pi}
		\int_{r(x)}^R\frac{ds}{\rho^2(s)}.
	\end{equation*}
	Indeed, $(\rho^2G_R')'=0$ away from the pole,
	$G_R=0$ on $\partial B_R(o)$, and $\rho(r)=r+O(r^3)$ gives the Euclidean
	singularity $G_R(o,x)=(4\pi r(x))^{-1}+O(1)$.
	Since
	$\rho(r)\sim\alpha r$ as $r\to\infty$, the monotone limit
	\begin{equation*}
		G(o,x)=\frac1{4\pi}
		\int_{r(x)}^\infty\frac{ds}{\rho^2(s)}
	\end{equation*}
	is finite away from $o$.
	It is the positive minimal Green function of $-\Delta_g$, so $\Sigma$ is
	nonparabolic.
	The same asymptotics show that
	$\Sigma$ has one end and cubic volume growth.
	
	\begin{proof}[Proof of \Cref{thm:counterexample}]
		Apply \Cref{prop:normal-derivative-deformation} to the equatorial slice above
		and to any nonzero compactly supported trace-free symmetric tensor $T$ on
		$\Sigma$.
		The resulting metric $G_\varepsilon$ equals $G_0$ outside a
		compact set and is uniformly equivalent to $G_0$ for small
		$|\varepsilon|$; hence it is complete.
		It has strictly positive sectional curvature everywhere.
		The induced metric on $\Sigma$ is unchanged, so completeness, one-endedness,
		nonparabolicity, and the diffeomorphism $\Sigma\cong\R^3$ are unchanged as
		well.
		The proposition
		preserves the stability form exactly and gives
		\begin{equation*}
			H_\varepsilon=0,
			\qquad
			A_\varepsilon=\varepsilon T\not\equiv0.
		\end{equation*}
		Thus $\Sigma$ is a complete, embedded, two-sided stable minimal hypersurface
		which is not totally geodesic.
		This proves the theorem.
	\end{proof}
	
	\bibliographystyle{amsalpha}
	\bibliography{reference}

\newcommand{\etalchar}[1]{$^{#1}$}
\providecommand{\bysame}{\leavevmode\hbox to3em{\hrulefill}\thinspace}
\providecommand{\MR}{\relax\ifhmode\unskip\space\fi MR }
\providecommand{\MRhref}[2]{%
  \href{http://www.ams.org/mathscinet-getitem?mr=#1}{#2}
}
\providecommand{\href}[2]{#2}
\begin{thebibliography}{CCM{\etalchar{+}}26b}

\bibitem[APX24]{APX}
Gioacchino Antonelli, Marco Pozzetta, and Kai Xu, \emph{A sharp spectral
  splitting theorem}, arXiv:2412.12707.

\bibitem[AX26]{AX}
Gioacchino Antonelli and Kai Xu, \emph{New spectral {Bishop--Gromov} and
  {Bonnet--Myers} theorems and applications to isoperimetry}, To appear in the
  Journal of the European Mathematical Society; arXiv:2405.08918, 2026.

\bibitem[CCE16]{Carlotto-Chodosh-Eichmair-PMT}
Alessandro Carlotto, Otis Chodosh, and Michael Eichmair, \emph{Effective
  versions of the positive mass theorem}, Invent. Math. \textbf{206} (2016),
  no.~3, 975--1016.

\bibitem[CCM{\etalchar{+}}26a]{CCMMR26}
Xavier Cabr{\'e}, Giovanni Catino, Luciano Mari, Paolo Mastrolia, and Alberto
  Roncoroni, \emph{Gradient estimates for the {Green} kernel under spectral
  {Ricci} bounds, and the stable {Bernstein} theorem in {$\mathbb R^4$}}, 2026,
  arXiv:2604.14393.

\bibitem[CCM{\etalchar{+}}26b]{cabre2026gradientestimatesgreenkernel}
Xavier Cabre, Giovanni Catino, Luciano Mari, Paolo Mastrolia, and Alberto
  Roncoroni, \emph{Gradient estimates for the green kernel under spectral ricci
  bounds, and the stable bernstein theorem in $\mathbb{R}^4$}.

\bibitem[CL23]{chodoshliR4anisotropic}
Otis Chodosh and Chao Li, \emph{Stable anisotropic minimal hypersurfaces in
  $\mathbb{R}^4$}, Forum Math. Pi \textbf{11} (2023), no.~e3, 1--22.

\bibitem[CL24]{ChodoshLi}
\bysame, \emph{Stable minimal hypersurfaces in {$\mathbb R^4$}}, Acta
  Mathematica \textbf{233} (2024), no.~1, 1--31.

\bibitem[CLS26]{CLS}
Otis Chodosh, Chao Li, and Douglas Stryker, \emph{Complete stable minimal
  hypersurfaces in positively curved four-manifolds}, Journal of the European
  Mathematical Society \textbf{28} (2026), no.~4, 1457--1487.

\bibitem[CMMR24]{CMMR}
Giovanni Catino, Luciano Mari, Paolo Mastrolia, and Alberto Roncoroni,
  \emph{Criticality, splitting theorems under spectral {Ricci} bounds and the
  topology of stable minimal hypersurfaces}, 2024, arXiv:2412.12631.

\bibitem[CMR24]{catino}
Giovanni Catino, Paolo Mastrolia, and Alberto Roncoroni, \emph{Two rigidity
  results for stable minimal hypersurfaces}, Geom. Funct. Anal. \textbf{34}
  (2024), no.~1, 1--18.

\bibitem[CS26]{ChaiSun}
Xiaoxiang Chai and Yukai Sun, \emph{Some rigidity theorems for spectral
  curvature bounds}, 2026, arXiv:2604.04052.

\bibitem[CV35]{CohnVossen}
Stefan Cohn-Vossen, \emph{{K{\"u}rzeste Wege und Totalkr{\"u}mmung auf
  Fl{\"a}chen}}, Compositio Mathematica \textbf{2} (1935), 69--133.

\bibitem[dCP79]{doCarmoPeng}
M.~do~Carmo and C.~K. Peng, \emph{Stable complete minimal surfaces in {$\mathbb
  R^3$} are planes}, Bulletin of the American Mathematical Society \textbf{1}
  (1979), no.~6, 903--906.

\bibitem[DMS26]{maximo-example}
PHILIPP~REISER DAVI~MAXIMO and DANIELE SEMOLA, \emph{Ricci curvature and
  minimal hypersurfaces with large betti numbers}, Calc. Var. Partial
  Differential Equations \textbf{65} (2026).

\bibitem[FCS80]{FCS}
Doris Fischer-Colbrie and Richard Schoen, \emph{The structure of complete
  stable minimal surfaces in three-manifolds of nonnegative scalar curvature},
  Communications on Pure and Applied Mathematics \textbf{33} (1980), no.~2,
  199--211.

\bibitem[Hon25]{hong24}
Han Hong, \emph{C{MC} hypersurface with finite index in hyperbolic space
  {$\Bbb{H}^4$}}, Adv. Math. \textbf{478} (2025), Paper No. 110408, 20.
  \MR{4924061}

\bibitem[HY24]{hongyan}
Han Hong and Zetian Yan, \emph{Rigidity of {CMC} hypersurfaces in 5- and
  6-manifolds}, 2024, arXiv:2405.06867.

\bibitem[MW24]{MW24}
Ovidiu Munteanu and Jiaping Wang, \emph{Bottom spectrum of three-dimensional
  manifolds with scalar curvature lower bound}, Journal of Functional Analysis
  \textbf{287} (2024), no.~2, 110457, 41 pp.

\bibitem[MW26]{MW}
\bysame, \emph{Bottom spectrum and parabolicity of three-manifolds with scalar
  curvature lower bound}, 2026, arXiv:2607.06508.

\bibitem[Pog81]{Pogorelov}
A.~V. Pogorelov, \emph{On the stability of minimal surfaces}, Soviet
  Mathematics. Doklady \textbf{24} (1981), 274--276, Translation of Dokl. Akad.
  Nauk SSSR 260 (1981), 293--295.

\bibitem[SY79]{Schoen-Yau-PSC}
Richard Schoen and Shing~Tung Yau, \emph{On the structure of manifolds with
  positive scalar curvature}, Manuscripta Math. \textbf{28} (1979), no.~1-3,
  159--183.

\bibitem[Wu23]{wuyujie}
Yujie Wu, \emph{Free boundary stable minimal hypersurfaces in positively curved
  4-manifolds}, 2023, arXiv:2308.08103.

\bibitem[Wu25]{wuyujie2}
\bysame, \emph{Rigidity of complete free boundary minimal hypersurfaces in
  convex {NNSC} manifolds}, 2025, arXiv:2504.20585.

\bibitem[Yau76]{Yau76}
Shing-Tung Yau, \emph{Some function-theoretic properties of complete
  {Riemannian} manifolds and their applications to geometry}, Indiana
  University Mathematics Journal \textbf{25} (1976), 659--670.

\bibitem[YZ26]{YanZhu}
Zetian Yan and Xinyu Zhu, \emph{Nonnegative {Ricci} curvature and uniformly
  convex boundary forces compactness}, 2026, arXiv:2605.31477.

\end{thebibliography}
	
\end{document}